\documentclass[a4paper,11pt]{article}

\usepackage{amsmath}
\usepackage{amssymb}
\usepackage{amsthm}
\usepackage{graphicx}
\usepackage{amscd}
\usepackage{epic, eepic}
\usepackage{url}
\usepackage{color}
\usepackage[utf8]{inputenc} 
\usepackage{enumerate}   
\usepackage{todonotes}
\usepackage{graphicx}
\usepackage{epstopdf}
\usepackage{enumitem}
\usepackage{tikz}
\usepackage[top=25mm, bottom=25mm, left=28mm, right=28mm]{geometry}
\usepackage[bookmarks=false,hidelinks]{hyperref}

\makeatletter
\renewenvironment{proof}[1][\proofname] {\par\pushQED{\qed}\normalfont\topsep6\p@\@plus6\p@\relax\trivlist\item[\hskip\labelsep\bfseries#1\@addpunct{.}]\ignorespaces}{\popQED\endtrivlist\@endpefalse}
\makeatother

\newtheorem{theorem}{\bf Theorem}[section]
\newtheorem{lemma}[theorem]{\bf Lemma}
\newtheorem{corollary}[theorem]{\bf Corollary}

\newtheorem{conjecture}[theorem]{\bf Conjecture}
\newtheorem{question}[theorem]{\bf Question}

\theoremstyle{definition}

\newtheorem{definition}[theorem]{\bf Definition}

\def\ex{\mathrm{ex}}
\def\d{\delta}
\def \cG{\mathcal{G}}

\def\e{\varepsilon}
\def\Hom{\mathrm{hom}}

\title{Rainbow Tur\'an number of even cycles, repeated patterns and blow-ups of cycles}

\author{Oliver Janzer\thanks{Department of Mathematics, ETH Z\"urich, Switzerland.
E-mail: {\tt oliver.janzer@math.ethz.ch}. Research supported by an ETH Z\"urich Postdoctoral Fellowship 20-1 FEL-35.}}

\date{}

\begin{document}

\maketitle

\begin{abstract}
	The rainbow Tur\'an number $\ex^*(n,H)$ of a graph $H$ is the maximum possible number of edges in a properly edge-coloured $n$-vertex graph with no rainbow subgraph isomorphic to $H$. We prove that for any integer $k\geq 2$, $\ex^*(n,C_{2k})=O(n^{1+1/k})$. This is tight and establishes a conjecture of Keevash, Mubayi, Sudakov and Verstra\"ete. We use the same method to prove several other conjectures in various topics. First, we prove that there exists a constant $c$ such that any properly edge-coloured $n$-vertex graph with more than $cn(\log n)^4$ edges contains a rainbow cycle. It is known that there exist properly edge-coloured $n$-vertex graphs with $\Omega(n\log n)$ edges which do not contain any rainbow cycle.
	
	Secondly, we show that in any proper edge-colouring of $K_n$ with $o(n^{\frac{r}{r-1}\cdot \frac{k-1}{k}})$ colours, there exist $r$ colour-isomorphic, pairwise vertex-disjoint copies of $C_{2k}$. This proves in a strong form a conjecture of Conlon and Tyomkyn, and a strenghtened version proposed by Xu, Zhang, Jing and Ge.
	
	Moreover, we answer a question of Jiang and Newman by showing that there exists a constant $c=c(r)$ such that any $n$-vertex graph with more than $cn^{2-1/r}(\log n)^{7/r}$ edges contains the $r$-blowup of an even cycle. Finally, we prove that the $r$-blowup of $C_{2k}$ has Tur\'an number $O(n^{2-\frac{1}{r}+\frac{1}{k+r-1}+o(1)})$, which can be used to disprove an old conjecture of Erd\H os and Simonovits.
\end{abstract}

\section{Introduction}

In this paper we develop a method that allows us to find cycles with suitable extra properties in graphs with sufficiently many edges. We give applications in three different areas, which are introduced in the next three subsections.

\subsection{Rainbow Tur\'an numbers}

For a family of graphs $\mathcal{H}$, the Tur\'an number (or extremal number) $\ex(n,\mathcal{H})$ is the maximum number of edges in an $n$-vertex graph which does not contain any $H\in \mathcal{H}$ as a subgraph. When $\mathcal{H}=\{H\}$, we write $\ex(n,H)$ for the same function. This function is determined asymptotically by the Erd\H os--Stone--Simonovits \cite{ES46,ESi66} theorem when $H$ has chromatic number at least $3$. However, for bipartite graphs $H$, even the order of magnitude of $\ex(n,H)$ is unknown in general. For example, a result of Bondy and Simonovits \cite{BS74} states that $\ex(n,C_{2k})=O(n^{1+1/k})$, but a matching lower bound is only known when $k\in \{2,3,5\}$.

A variant of this function was introduced by Keevash, Mubayi, Sudakov and Verstra\"ete in \cite{KMSV03}. In an edge-coloured graph, we say that a subgraph is rainbow if all its edges are of different colour. The rainbow Tur\'an number of the graph $H$ is then defined to be the maximum number of edges in a properly edge-coloured $n$-vertex graph that does not contain a rainbow $H$ as a subgraph. This number is denoted by $\ex^*(n,H)$. Clearly, $\ex^*(n,H)\geq \ex(n,H)$ for every $n$ and $H$. Keevash, Mubayi, Sudakov and Verstra\"ete proved, among other things, that for any non-bipartite graph $H$, we have $\ex^*(n,H)=(1+o(1))\ex(n,H)$. Hence, the most challenging case again seems to be when $H$ is bipartite. Keevash et al. \cite{KMSV03} showed that $\ex^*(n,K_{s,t})=O(n^{2-1/s})$, which is tight when $t>(s-1)!$ \cite{KRS96,ARS99}. The function has also been studied for trees (see \cite{JPS17,EGM19,JR19}). About even cycles, Keevash et al. proved the following lower bound.

\begin{theorem}[Keevash--Mubayi--Sudakov--Verstra\"ete \cite{KMSV03}]
	For any integer $k\geq 2$,
	$$\ex^*(n,C_{2k})=\Omega(n^{1+1/k}).$$
\end{theorem}

They conjectured that this is tight.

\begin{conjecture}[Keevash--Mubayi--Sudakov--Verstra\"ete \cite{KMSV03}] \label{conj:rainbowc2k}
	For any integer $k\geq 2$,
	$$\ex^*(n,C_{2k})=\Theta(n^{1+1/k}).$$
\end{conjecture}

They have verified their conjecture for $k\in \{2,3\}$. For general $k$, Das, Lee and Sudakov proved the following upper bound.

\begin{theorem}[Das--Lee--Sudakov \cite{DLS13}]
	For every fixed integer $k\geq 2$, $$\ex^*(n,C_{2k})=O\left(n^{1+\frac{(1+\e_k)\ln k}{k}}\right),$$ where $\e_k\rightarrow 0$ as $k\rightarrow \infty$.
\end{theorem}

In this paper we prove Conjecture \ref{conj:rainbowc2k} by establishing the following result.

\begin{theorem} \label{thm:rainbowC2k}
	For any integer $k\geq 2$, we have
	$$\ex^*(n,C_{2k})=O(n^{1+1/k}).$$
\end{theorem}

The theta graph $\theta_{k,t}$ is the union of $t$ paths of length $k$ which share the same endpoints but are pairwise internally vertex-disjoint. Note that $\theta_{k,2}=C_{2k}$. Our proof works also for general theta graphs.

\begin{theorem} \label{thm:rainbowtheta}
	For any integers $k,t\geq 2$, we have $$\ex^*(n,\theta_{k,t})=O(n^{1+1/k}).$$
\end{theorem}

\noindent This generalises a result of Faudree and Simonovits \cite{FS83} stating that $\ex(n,\theta_{k,t})=O(n^{1+1/k})$.
	
Keevash, Mubayi, Sudakov and Verstra\"ete also asked how many edges a properly edge-coloured $n$-vertex graph can have if it does not contain any rainbow cycle. They constructed such graphs with $\Omega(n\log n)$ edges. Note that this is quite different from the uncoloured case, since any $n$-vertex acyclic graph has at most $n-1$ edges. Das, Lee and Sudakov proved that if $\eta>0$ and $n$ is sufficiently large, then any properly edge-coloured $n$-vertex graph with at least $n\exp\left((\log n)^{\frac{1}{2}+\eta}\right)$ edges contains a rainbow cycle. We prove the following improvement.
	
\begin{theorem} \label{thm:acyclicrainbow}
	There exists an absolute constant $C$ such that if $n$ is sufficiently large and $G$ is a properly edge-coloured graph on $n$ vertices with at least $Cn(\log n)^4$ edges, then $G$ contains a rainbow cycle of even length.
\end{theorem}

\subsection{Colour-isomorphic even cycles in proper colourings}

Conlon and Tyomkyn \cite{CT20} have initiated the study of the following problem. We say that two subgraphs of an edge-coloured graph are colour-isomorphic if there is an isomorphism between them preserving the colours. For an integer $r\geq 2$ and a graph $H$, they write $f_r(n,H)$ for the smallest number $C$ so that there is a proper edge-colouring of $K_n$ with $C$ colours containing no $r$ vertex-disjoint colour-isomorphic copies of $H$. They proved various general results about this function, such as the following upper bound.

\begin{theorem}[Conlon--Tyomkyn \cite{CT20}] \label{thm:ConlonTyomkyn}
	For any graph $H$ with $v$ vertices and $e$ edges,
	$$f_r(n,H)=O\left(\max\left(n,n^{\frac{rv-2}{(r-1)e}}\right)\right).$$
\end{theorem}

Regarding even cycles, they established the following result.

\begin{theorem}[Conlon--Tyomkyn \cite{CT20}]
	$f_2(n,C_6)=\Omega(n^{4/3})$.
\end{theorem}

One of the several open problems they posed is the following question.

\begin{question}[Conlon--Tyomkyn \cite{CT20}]
	Is it true that for every $\e>0$, there exists $k_0=k_0(\e)$ such that, for all $k\geq k_0$, $f_2(n,C_{2k})=\Omega(n^{2-\e})$?
\end{question}

Later, Xu, Zhang, Jing and Ge made a more precise conjecture.

\begin{conjecture}[Xu--Zhang--Jing--Ge \cite{XZJG20}]
	For any $k\geq 3$,
	$f_2(n,C_{2k})=\Omega(n^{2-\frac{2}{k}})$.
\end{conjecture}

We prove this conjecture in a more general form.

\begin{theorem} \label{thm:repeatedcycles}
	Let $k,r\geq 2$ be fixed integers. Then $f_r(n,C_{2k})=\Omega\left(n^{\frac{r}{r-1}\cdot \frac{k-1}{k}}\right)$.
\end{theorem}

\subsection{Tur\'an number of blow-ups of cycles}

For a graph $F$, the $r$-blowup of $F$ is the graph obtained by replacing each vertex of $F$ with an independent set of size $r$ and each edge of $F$ by a $K_{r,r}$. We write $F[r]$ for this graph. The systematic study of the Tur\'an number of blow-ups was initiated by Grzesik, Janzer and Nagy \cite{GJN19}. They proved that for any tree $T$ we have $\ex(n,T[r])=O(n^{2-1/r})$. They have also made the following general conjecture.

\begin{conjecture}[Grzesik--Janzer--Nagy \cite{GJN19}] \label{con:turanblowup}
	Let $r$ be a positive integer and let $F$ be a graph such that $\ex(n,F)=O(n^{2-\alpha})$ for some $0\leq \alpha\leq 1$ constant. Then
	$$\ex(n,F[r])=O(n^{2-\frac{\alpha}{r}}).$$
\end{conjecture}

Their result mentioned above proves this conjecture when $F$ is a tree. It is easy to see that the conjecture holds also when $F=K_{s,t}$ and $\alpha=1/s$.

In the case of forbidding all $r$-blowups of cycles, an earlier question was formulated by Jiang and Newman \cite{JN17}. To state this question, we write $\mathcal{C}[r]=\{C_{2k}[r]: k\geq 2\}$.

\begin{question}[Jiang--Newman \cite{JN17}]
	Is it true that for any positive integer $r$ and any $\e>0$, $\ex(n,\mathcal{C}[r])=O(n^{2-\frac{1}{r}+\e})$?
\end{question}

We answer this question affirmatively in a stronger form.

\begin{theorem} \label{thm:turanblowup}
	For any positive integer $r$,
	$$\ex(n,\mathcal{C}[r])=O(n^{2-1/r}(\log n)^{7/r}).$$
\end{theorem}

Binomial random graphs show that $\ex(n,\mathcal{C}[r])=\Omega(n^{2-1/r})$. It would be interesting to decide whether the logarithmic factor in Theorem \ref{thm:turanblowup} can be removed.

Finally, we establish an upper bound for the Tur\'an number when only one blownup cycle is forbidden.

\begin{theorem} \label{thm:turanblowup2k}
	For any integers $r\geq 1$ and $k\geq 2$, we have
	$$\ex(n,C_{2k}[r])=O\left(n^{2-\frac{1}{r}+\frac{1}{k+r-1}}(\log n)^{\frac{4k}{r(k+r-1)}}\right).$$
\end{theorem}

This is still quite a long way from the conjectured $\ex(n,C_{2k}[r])=O(n^{2-\frac{1}{r}+\frac{1}{kr}})$. However, it can be used to disprove the following conjecture of Erd\H os and Simonovits.

\begin{conjecture}[Erd\H os--Simonovits \cite{ESonline}] \label{con:nondegenerate}
	Let $H$ be a bipartite graph with minimum degree~$s$. Then there exists $\e>0$ such that $\ex(n,H)=\Omega(n^{2-\frac{1}{s-1}+\e})$.
\end{conjecture}

To see that this is false, note that the graph $C_{2k}[r]$ has minimum degree $2r$, but, by Theorem \ref{thm:turanblowup2k}, for any $\d>0$, we have $\ex(n,C_{2k}[r])=O(n^{2-\frac{1}{r}+\d})$ for sufficiently large $k$. This means that there exists, for any even $s$ and any $\d>0$, a bipartite graph $H$ with minimum degree $s$ which has $\ex(n,H)=O(n^{2-\frac{2}{s}+\d})$, disproving Conjecture \ref{con:nondegenerate} for all even $s\geq 4$. (See the concluding remarks for a brief discussion of the odd case.) On the other hand, a simple application of the probabilistic deletion method shows that if $H$ is a bipartite graph with minimum degree $s\geq 2$, then there exists $\e=\e(H)>0$ such that $\ex(n,H)=\Omega(n^{2-\frac{2}{s}+\e})$.

\subsection{Overview of the proof}

Here we give a brief sketch of the proof of Theorem \ref{thm:rainbowC2k}. The proofs of the other results will be somewhat similar and in the paper we will use a unified approach. A graph homomorphism from $H$ to $G$ is a map $f:V(H)\rightarrow V(G)$ such that whenever $uv\in E(H)$, we have $f(u)f(v)\in E(G)$. Write $\hom(H,G)$ for the number of such maps. For a properly edge-coloured graph $G$, we call a homomorphism from $C_{2k}$ to $G$ rainbow if the images of different edges of $C_{2k}$ have different colour in $G$. Our key contribution is an upper bound on the number of homomorphisms $C_{2k}\rightarrow G$ which are not rainbow in terms of the maximum degree of $G$, $\hom(C_{2k},G)$ and, perhaps surprisingly, $\hom(C_{2k-2},G)$. Then using a standard inequality between $\hom(C_{2k},G)$ and $\hom(C_{2k-2},G)$, we can turn the bound into one which only involves $|V(G)|$, $\Delta(G)$ and $\hom(C_{2k},G)$, and in which $\hom(C_{2k},G)$ has exponent $1-\frac{1}{2k}$. This means that if $\hom(C_{2k},G)$ is large compared to $|V(G)|$ and $\Delta(G)$ (more precisely, if $\hom(C_{2k},G)=\omega(|V(G)|\Delta(G)^k)$), then there is a homomorphism from $C_{2k}$ to $G$ which is rainbow, essentially giving us a rainbow $2k$-cycle. Given a properly edge-coloured $n$-vertex graph $G$ with $\omega(n^{1+1/k})$ edges, we use a standard result to obtain an $m$-vertex subgraph $G'$ with $\omega(m^{1+1/k})$ edges which is nearly regular. In particular, the average degree of $G'$ is $\omega(m^{1/k})$. It will not be hard to see that $\hom(C_{2k},G')=\omega(m\Delta(G')^k)$, hence we can use our upper bound on the number of non-rainbow homomorphisms to show that there exists a rainbow homomorphism $C_{2k}\rightarrow G'$.

\medskip

The rest of this paper is organised as follows. In Section \ref{sec:mainlemma}, we establish our key lemma in various forms. Most of our results will follow fairly straightforwardly from these lemmas. In Section \ref{sec:2kcycle}, we prove Theorems \ref{thm:rainbowC2k}, \ref{thm:rainbowtheta} and \ref{thm:repeatedcycles}. In Section \ref{sec:acyclic}, we prove Theorem \ref{thm:acyclicrainbow}. The proofs of Theorem \ref{thm:turanblowup} and \ref{thm:turanblowup2k} are given in Section \ref{sec:blownupcycles}. We make some concluding remarks and mention open problems in Section \ref{sec:rainbowremarks}.

\medskip

\textbf{Notation.} As mentioned above, we write $\Hom(H,G)$ for the number of graph homomorphisms $V(H)\rightarrow V(G)$. $P_k$ will denote the path with $k$ edges and we use the convention that $C_2=P_1$ and $C_0$ is the one-vertex graph. For vertices $x,y\in V(G)$, $\Hom_{x,y}(P_k,G)$ denotes the number of walks of length $k$ in $G$ from $x$ to $y$. We will write $d_G(x)$ (or $d(x)$ if $G$ is clear) for the degree of the vertex $x$ in $G$ and we write $N_G(x)$ or $N(x)$ for the neighbourhood of $x$. Finally, we write $\d(G)$ and $\Delta(G)$ for the minimum and maximum degree of $G$, respectively. Logarithms are base $2$ unless stated otherwise.

\section{The key lemma} \label{sec:mainlemma}

Our goal in this section is to develop a method for finding cycles of given length avoiding certain `degenerate' properties. Roughly speaking, our main lemma will be an upper bound on the number of those cycles which are not suitable. We have been deliberately vague about what we mean by a `suitable' cycle. In our first application it will mean rainbow cycle, but later it will have different meanings. For example, in Section \ref{sec:blownupcycles} we will work in auxiliary graphs whose vertices are sets, and a `suitable' cycle in these auxiliary graphs will be one whose vertices are disjoint sets.

With a slight abuse of terminology, we call a homomorphism $H\rightarrow G$ a \emph{homomorphic copy of $H$ in $G$}. That is, a homomorphic copy of $C_{2k}$ is a tuple $(x_1,\dots,x_{2k})\in V(G)^{2k}$ such that $x_i x_{i+1}\in E(G)$ for every $1\leq i\leq 2k$. Here and below, the indices are understood modulo $2k$, i.e. $x_{2k+1}=x_1$. 

\subsection{The simplest versions}

We first state our key lemma in two simple forms which will nevertheless be sufficient for most of our applications. The first one says, roughly speaking, that if `conflicting edges' are locally rare in a graph, then most of the homomorphic $2k$-cycles do not contain a conflicting pair of edges.

\begin{lemma} \label{lemma:simple with edges}
	Let $k\geq 2$ be an integer and let $G=(V,E)$ be a graph on $n$ vertices. Let $\sim$ be a symmetric binary relation defined over $E$ such that for every $uv\in E$ and $w\in V$, $w$ has at most $s$ neighbours $z\in V$ which satisfy $uv\sim wz$.
	Then the number of homomorphic $2k$-cycles $(x_1,x_2,\dots,x_{2k})$ in $G$ such that $x_ix_{i+1}\sim x_jx_{j+1}$ for some $i\neq j$ (here and below $x_{2k+1}:=x_1$) is at most $$32k^{3/2}s^{1/2}\Delta(G)^{1/2}n^{\frac{1}{2k}}\hom(C_{2k},G)^{1-\frac{1}{2k}}.$$
\end{lemma}

For example, when our goal is to find rainbow cycles, we set $e\sim f$ when the edges $e$ and $f$ have the same colour and apply the lemma to conclude that if $\hom(C_{2k},G)=\omega(n\Delta(G)^k)$, then most of the homomorphic $2k$-cycles in $G$ are rainbow. We will also frequently use the following variant which bounds the number of homomorphic cycles with conflicting \emph{vertices}, rather than edges.

\begin{lemma} \label{lemma:simple with vertices}
	Let $k\geq 2$ be an integer and let $G=(V,E)$ be a graph on $n$ vertices. Let $\sim$ be a symmetric binary relation defined over $V$ such that for every $u\in V$ and $v\in V$, $v$ has at most $s$ neighbours $w\in V$ which satisfy $u\sim w$.
	Then the number of homomorphic $2k$-cycles $(x_1,x_2,\dots,x_{2k})$ in $G$ such that $x_i\sim x_j$ for some $i\neq j$ is at most $$32k^{3/2}s^{1/2}\Delta(G)^{1/2}n^{\frac{1}{2k}}\hom(C_{2k},G)^{1-\frac{1}{2k}}.$$
\end{lemma}

When we aim to find rainbow cycles, we will take $u\sim v$ if and only if $u=v$. This will allow us to conclude that if $\hom(C_{2k},G)=\omega(n\Delta(G)^k)$, then almost every homomorphism $C_{2k}\rightarrow G$ is injective and hence its image is isomorphic to $C_{2k}$.

\subsection{The general statement and its variations}

We now state a more technical version of the key lemma which is sometimes more convenient when our graph is bipartite with unbalanced parts. In this paper we will only apply the lemma with $X_1=X_2=V$ or with $X_1$ and $X_2$ partitioning $V$. Another difference compared to the lemmas from the previous subsection is that the formula involves $\hom(C_{2\ell},G)$. Using $\hom(C_0,G)=|V(G)|$, we recover the bound in the previous lemmas when $\ell=0$. In this paper, we will always take $\ell=0$ or $\ell=1$ in our applications (since it is more difficult to find good upper bounds for $\hom(C_{2\ell},G)$ when $\ell$ is bigger).

\begin{lemma} \label{lemma:bipartite with 2l}
	Let $k\geq 2$ and $0\leq \ell\leq k-1$ be integers and let $G=(V,E)$ be a graph on $n$ vertices. Let $X_1$ and $X_2$ be subsets of $V$. Let $\sim$ be a symmetric binary relation defined over $V^2$ such that
	\begin{itemize}
		\item for every $(u,v)\in V^2$ and $w\in X_1$, $w$ has at most $\Delta_1$ neighbours $z\in X_2$ and amongst them at most $s_1$ satisfies $(u,v)\sim (z,w)$, and
		\item for every $(u,v)\in V^2$ and $w\in X_2$, $w$ has at most $\Delta_2$ neighbours $z\in X_1$ and amongst them at most $s_2$ satisfies $(u,v)\sim (z,w)$.
	\end{itemize}
	Let $M=\max(\Delta_1 s_2,\Delta_2 s_1)$.
	Then the number of homomorphic $2k$-cycles $(x_1,x_2,\dots,x_{2k})\in (X_1\times X_2\times X_1\times \dots \times X_2)\cup (X_2\times X_1\times X_2\times \dots \times X_1)$ in $G$ such that $(x_i,x_{i+1})\sim (x_j,x_{j+1})$ for some $i\neq j$ is at most $$32k^{3/2}M^{1/2}\hom(C_{2\ell},G)^{\frac{1}{2(k-\ell)}}\hom(C_{2k},G)^{1-\frac{1}{2(k-\ell)}}.$$
\end{lemma}

We will again need a version of this which counts the homomorphisms with conflicting vertices.

\begin{lemma} \label{lemma:bipartite with vertices}
	Let $k\geq 2$ and $0\leq \ell\leq k-1$ be integers and let $G=(V,E)$ be a graph on $n$ vertices. Let $X_1$ and $X_2$ be subsets of $V$. Let $\sim$ be a symmetric binary relation defined over $V$ such that
	\begin{itemize}
		\item for every $u\in V$ and $v\in X_1$, $v$ has at most $\Delta_1$ neighbours $w\in X_2$ and amongst them at most $s_1$ satisfies $u\sim w$, and
		\item for every $u\in V$ and $v\in X_2$, $v$ has at most $\Delta_2$ neighbours $w\in X_1$ and amongst them at most $s_2$ satisfies $u\sim w$.
	\end{itemize}
	Let $M=\max(\Delta_1 s_2,\Delta_2 s_1)$.
	Then the number of homomorphic $2k$-cycles $(x_1,x_2,\dots,x_{2k})\in (X_1\times X_2\times X_1\times \dots \times X_2)\cup (X_2\times X_1\times X_2\times \dots \times X_1)$ in $G$ such that $x_i\sim x_j$ for some $i\neq j$ is at most $$32k^{3/2}M^{1/2}\hom(C_{2\ell},G)^{\frac{1}{2(k-\ell)}}\hom(C_{2k},G)^{1-\frac{1}{2(k-\ell)}}.$$
\end{lemma}

Finally, we state the lemma from which all the above will follow. It is Lemma \ref{lemma:bipartite with 2l} in the special case $\ell=k-1$.

\begin{lemma} \label{lemma:most general}
	Let $k\geq 2$ be an integer and let $G=(V,E)$ be a graph. Let $X_1$ and $X_2$ be subsets of $V$. Let $\sim$ be a symmetric binary relation defined over $V^2$ such that
	\begin{itemize}
		\item for every $(u,v)\in V^2$ and $w\in X_1$, $w$ has at most $\Delta_1$ neighbours $z\in X_2$ and amongst them at most $s_1$ satisfies $(u,v)\sim (z,w)$, and
		\item for every $(u,v)\in V^2$ and $w\in X_2$, $w$ has at most $\Delta_2$ neighbours $z\in X_1$ and amongst them at most $s_2$ satisfies $(u,v)\sim (z,w)$.
	\end{itemize}
	Let $M=\max(\Delta_1 s_2,\Delta_2 s_1)$.
	Then the number of homomorphic $2k$-cycles $(x_1,x_2,\dots,x_{2k})\in (X_1\times X_2\times X_1\times \dots \times X_2)\cup (X_2\times X_1\times X_2\times \dots \times X_1)$ in $G$ such that $(x_i,x_{i+1})\sim (x_j,x_{j+1})$ for some $i\neq j$ is at most $$32k\left(kM\hom(C_{2k-2},G)\hom(C_{2k},G)\right)^{1/2}.$$
\end{lemma}

\subsection{The proofs}

We first prove Lemma \ref{lemma:most general} and later deduce the other variants from it.

\begin{proof}[Proof of Lemma \ref{lemma:most general}]
	For a positive integer $r$, let $\alpha_r$ be the number of walks of length $k-1$ in $G$ whose endpoints $y$ and $z$ have $2^{r-1}\leq \hom_{y,z}(P_{k-1},G)<2^r$ and let $\beta_r$ be the number of walks of length $k$ in $G$ whose endpoints $y$ and $z$ have $2^{r-1}\leq \hom_{y,z}(P_k,G)<2^r$. Clearly,
	\begin{equation}
	\sum_{r\geq 1} \alpha_r 2^{r-1}\leq \Hom(C_{2k-2},G) \label{eqn:alphahom}
	\end{equation} and
	\begin{equation}
	\sum_{r\geq 1} \beta_r 2^{r-1}\leq \Hom(C_{2k},G). \label{eqn:betahom}
	\end{equation}
	
	\sloppy For positive integers $r$ and $t$, write $\gamma_{r,t}$ for the number of homomorphic $2k$-cycles $(x_1,x_2,\dots ,x_{2k})\in X_1\times X_2\times X_1\times \dots \times X_2$ such that $2^{r-1}\leq \hom_{x_1,x_{k+2}}(P_{k-1},G)<2^r$, $2^{t-1}\leq \hom_{x_2,x_{k+2}}(P_k,G)<2^t$ and there exists $2\leq i\leq k+1$ for which $(x_1,x_2)\sim (x_i,x_{i+1})$.
	
	\medskip
	
	\noindent \textbf{Claim 1.}
	\begin{equation}
		\gamma_{r,t}\leq \alpha_r \cdot \Delta_1 \cdot 2^t. \label{eqn:gammast first}
	\end{equation}
	\noindent \textbf{Proof of Claim 1.}
	If $(x_1,x_2,\dots ,x_{2k})\in X_1\times X_2\times X_1\times \dots \times X_2$ is a homomorphic $C_{2k}$ with $2^{r-1}\leq \hom_{x_1,x_{k+2}}(P_{k-1},G)<2^r$ and $2^{t-1}\leq \hom_{x_2,x_{k+2}}(P_k,G)<2^t$, then there are at most $\alpha_r$ ways to choose $(x_{k+2},x_{k+3},\dots,x_{2k},x_1)$, given such a choice there are at most $\Delta_1$ choices for $x_2$, and given these there are at most $2^t$ choices for $(x_3,\dots,x_{k+1})$. $\Box$
	
	\medskip
	
	\noindent \textbf{Claim 2.}
	\begin{equation}
		\gamma_{r,t}\leq \beta_t \cdot ks_2 \cdot 2^r. \label{eqn:gammast second}
	\end{equation}
	\noindent \textbf{Proof of Claim 2.}
	If $(x_1,x_2,\dots ,x_{2k})\in X_1\times X_2\times X_1\times \dots \times X_2$ is a homomorphic $C_{2k}$ with $2^{r-1}\leq \hom_{x_1,x_{k+2}}(P_{k-1},G)<2^r$ and $2^{t-1}\leq \hom_{x_2,x_{k+2}}(P_k,G)<2^t$, and there exists $2\leq i\leq k+1$ such that $(x_1,x_2)\sim (x_i,x_{i+1})$, then there are at most $\beta_t$ ways to choose $(x_2,\dots,x_{k+2})$. Given such a choice, there are at most $ks_2$ possibilities for $x_1$, since $(x_1,x_2)\sim (x_i,x_{i+1})$ for some $2\leq i\leq k+1$ and the pairs $(x_2,x_3),\dots,(x_{k+1},x_{k+2})$ are already fixed. Finally, there are at most $2^r$ ways to complete this to a suitable homomorphic copy of $C_{2k}$ by choosing $(x_{2k},x_{2k-1},\dots,x_{k+3})$. $\Box$
	
	\medskip
	
	The total number of homomorphic copies $(x_1,x_2,\dots ,x_{2k})\in X_1\times X_2\times X_1\times \dots \times X_2$ of $C_{2k}$ such that there exists $2\leq i\leq k+1$ with $(x_1,x_2)\sim (x_i,x_{i+1})$ is $\sum_{r,t\geq 1} \gamma_{r,t}$. We give an upper bound for this sum as follows. Let  $q$ be the integer for which $(\frac{ks_2 \hom(C_{2k},G)}{\Delta_1\hom(C_{2k-2},G)})^{1/2}\leq 2^q<2(\frac{ks_2\hom(C_{2k},G)}{\Delta_1\hom(C_{2k-2},G)})^{1/2}$.
	Now, using equations (\ref{eqn:gammast first}) and (\ref{eqn:alphahom}),
	\begin{align*}
	\sum_{r,t:t< r+q} \gamma_{r,t}&\leq \Delta_1\sum_{r,t:t< r+q} 2^t\alpha_r\leq \Delta_1 \sum_{r\geq 1} 2^{r+q}\alpha_r \leq \Delta_1 2^{q+1}\Hom(C_{2k-2},G) \\
	&\leq 4(ks_2 \Delta_1\hom(C_{2k-2},G)\hom(C_{2k},G))^{1/2}.
	\end{align*}
	Also, using equations (\ref{eqn:gammast second}) and (\ref{eqn:betahom}),
	\begin{align*}
	\sum_{r,t:t\geq r+q} \gamma_{r,t}
	&\leq ks_2 \sum_{r,t:t\geq r+q} 2^{r}\beta_t\leq ks_2 \cdot \sum_{t\geq 1} 2^{t-q+1}\beta_t\leq ks_2 \cdot 2^{-q+2}\Hom(C_{2k},G) \\
	&\leq 4(ks_2 \Delta_1\hom(C_{2k-2},G)\hom(C_{2k},G))^{1/2}.
	\end{align*}
	Thus, the number of homomorphic $2k$-cycles $(x_1,x_2,\dots ,x_{2k})\in X_1\times X_2\times X_1\times \dots \times X_2$ with $(x_1,x_2)\sim (x_i,x_{i+1})$ for some $2\leq i\leq k+1$ is
	$$\sum_{r,t\geq 1} \gamma_{r,t}\leq 8(ks_2 \Delta_1\hom(C_{2k-2},G)\hom(C_{2k},G))^{1/2}\leq 8(kM\hom(C_{2k-2},G)\hom(C_{2k},G))^{1/2}.$$
	By the symmetry of $X_1$ and $X_2$, the number of homomorphic $2k$-cycles $(x_1,x_2,\dots ,x_{2k})\in (X_1\times X_2\times X_1\times \dots \times X_2)\cup (X_2\times X_1\times X_2\times \dots \times X_1)$ with $(x_1,x_2)\sim (x_i,x_{i+1})$ for some $2\leq i\leq k+1$ is at most $16(kM\hom(C_{2k-2},G)\hom(C_{2k},G))^{1/2}$.
	The statement of the lemma follows by cyclic symmetry.
\end{proof}

To prove Lemma \ref{lemma:bipartite with 2l}, we use the next two inequalities.

\begin{lemma} \label{lemma:hom inequality}
	For any integer $\ell\geq 2$ and any non-empty graph $G$, we have
	$$\frac{\hom(C_{2\ell},G)}{\hom(C_{2\ell-2},G)}\geq \frac{\hom(C_{2\ell-2},G)}{\hom(C_{2\ell-4},G)}.$$
\end{lemma}

\begin{proof}
	Let $\lambda_1,\lambda_2,\dots,\lambda_n$ be the eigenvalues of $G$. It is well-known (see, for example, Corollary~1.3 in \cite{St13}) that
	$$\hom(C_{2j},G)=\sum_{i=1}^n \lambda_i^{2j}$$
	for every positive integer $j$. By the Cauchy-Schwartz inequality, we have
	$$\left(\sum_{i=1}^n \lambda_i^{2\ell}\right)\left(\sum_{i=1}^n \lambda_i^{2\ell-4}\right)\geq \left(\sum_{i=1}^n \lambda_i^{2\ell-2}\right)^2.$$
	Hence,
	$$\hom(C_{2\ell},G)\hom(C_{2\ell-4},G)\geq \hom(C_{2\ell-2},G)^2.$$
	The result follows after rearrangement.
\end{proof}

\begin{corollary} \label{cor:hom inequality}
	For any integers $k\geq 2$ and $0\leq \ell\leq k-1$ and any graph $G$,
	$$\hom(C_{2k-2},G)\leq \hom(C_{2\ell},G)^{\frac{1}{k-\ell}}\hom(C_{2k},G)^{1-\frac{1}{k-\ell}}.$$
\end{corollary}

\begin{proof}
	By Lemma \ref{lemma:hom inequality}, we have $$\frac{\hom(C_{2k},G)}{\hom(C_{2\ell},G)}=\frac{\hom(C_{2k},G)}{\hom(C_{2k-2},G)}\dots \frac{\hom(C_{2\ell+2},G)}{\hom(C_{2\ell},G)}\leq \left(\frac{\hom(C_{2k},G)}{\hom(C_{2k-2},G)}\right)^{k-\ell}.$$
	The result follows after rearrangement.
\end{proof}

Lemma \ref{lemma:bipartite with 2l} follows immediately from Lemma \ref{lemma:most general} using Corollary \ref{cor:hom inequality}.

\begin{proof}[Proof of Lemma \ref{lemma:bipartite with vertices}]
	Define a symmetric binary relation $\sim'$ over $V^2$ by setting $(u,v)\sim' (z,w)$ if and only if $u\sim z$. By the assumption on $\sim$, for any $(u,v)\in V^2$ and $w\in X_1$, $w$ has at most $\Delta_1$ neighbours $z\in X_2$ and amongst them at most $s_1$ satisfies $u\sim z$, and hence $w$ has at most $s_1$ neighbours $z\in X_2$ with $(u,v)\sim' (z,w)$. By a nearly identical argument, for any $(u,v)\in V^2$ and $w\in X_2$, $w$ has at most $\Delta_2$ neighbours $z\in X_1$ and amongst them at most $s_2$ satisfies $(u,v)\sim' (z,w)$. Therefore, by Lemma \ref{lemma:bipartite with 2l}, the number of homomorphic $2k$-cycles $(x_1,x_2,\dots,x_{2k})\in (X_1\times X_2\times X_1\times \dots \times X_2)\cup (X_2\times X_1\times X_2\times \dots \times X_1)$ in $G$ such that $(x_i,x_{i+1})\sim' (x_j,x_{j+1})$ for some $i\neq j$ is at most $$32k^{3/2}M^{1/2}\hom(C_{2\ell},G)^{\frac{1}{2(k-\ell)}}\hom(C_{2k},G)^{1-\frac{1}{2(k-\ell)}}.$$ But this is the same as counting the homomorphic $2k$-cycles $(x_1,x_2,\dots,x_{2k})\in (X_1\times X_2\times X_1\times \dots \times X_2)\cup (X_2\times X_1\times X_2\times \dots \times X_1)$ in $G$ such that $x_i\sim x_j$ for some $i\neq j$, completing the proof.
\end{proof}

To deduce Lemma \ref{lemma:simple with vertices}, take $\ell=0$ and $X_1=X_2=V$ in Lemma \ref{lemma:bipartite with vertices} and note that $\hom(C_0,G)=|V(G)|$. Finally, Lemma~\ref{lemma:simple with edges} follows from applying Lemma \ref{lemma:bipartite with 2l} with $\ell=0$, $X_1=X_2=V$ and $\sim'$ instead of $\sim$ where $(u,v)\sim' (z,w)$ if and only if $uv\sim zw$.

\section{Finding a good $2k$-cycle} \label{sec:2kcycle}

In this section we prove a theorem which is a substantial generalization of the Bondy--Simonovits result $\ex(n,C_{2k})=O(n^{1+1/k})$ in that it allows us to find $2k$-cycles with extra properties provided our $n$-vertex graph has at least $Cn^{1+1/k}$ edges. The result will immediately imply Theorem~\ref{thm:rainbowC2k} and will be used to give a short proof of Theorem~\ref{thm:repeatedcycles}. Unlike the key lemmas from the previous section, this theorem is only useful when we want to find a cycle of fixed length, and accordingly it will not be used in the proofs of Theorem~\ref{thm:acyclicrainbow} and Theorem~\ref{thm:turanblowup}.

\begin{theorem} \label{thm:general 2kcycle}
	Let $k\geq 2$ and $s$ be positive integers. Then there exists a constant $C=C(k,s)$ with the following property. Suppose that $G=(V,E)$ is a graph with $n$ vertices and at least $Cn^{1+1/k}$ edges. Let $\sim$ be a symmetric binary relation on $V$ such that for every $u\in V$ and $v\in V$, $v$ has at most $s$ neighbours $w\in V$ which satisfy $u\sim w$. Let $\approx$ be a binary relation on $E$ such that for every $uv\in E$ and $w\in V$, $w$ has at most $s$ neighbours $z\in V$ which satisfy $uv\approx wz$. Then $G$ contains a $2k$-cycle $x_1x_2\dots x_{2k}$ such that $x_i\not \sim x_j$ for every $i\neq j$ and $x_ix_{i+1}\not \approx x_jx_{j+1}$ for every $i\neq j$ (where, as before, $x_{2k+1}:=x_1$).
\end{theorem}

As we will see in a moment, this theorem follows fairly easily from the key lemmas in the previous section via a standard regularization argument. We will pass to a subgraph which has many homomorphic $2k$-cycles, but whose maximum degree is not too large.
The large number of homomorphism will be guaranteed by the following lemma, due to Sidorenko.

\begin{lemma}[Sidorenko \cite{Si92}] \label{lemma:sidorenko}
	Let $k\geq 2$ be an integer and let $G$ be a graph on $n$ vertices. Then $$\hom(C_{2k},G)\geq \left(\frac{2|E(G)|}{n}\right)^{2k}.$$
\end{lemma}

The lemma equivalently states that $\hom(C_{2k},G)\geq \bar{d}^{2k}$, where $\bar{d}$ is the average degree of $G$. 

We say that a graph $G$ is $K$-almost regular if $\Delta(G)\leq K\delta(G)$. We will use the following lemma of Jiang and Seiver, which is a slight modification of a much earlier result by Erd\H os and Simonovits \cite{ES70}.

\begin{lemma}[Jiang--Seiver \cite{JS12}] \label{lemmaJS}
	Let $\e,c$ be positive reals, where $\e<1$ and $c\geq 1$. Let $n$ be a positive integer that is sufficiently large as a function of $\e$. Let $G$ be a graph on $n$ vertices with $e(G)\geq cn^{1+\e}$. Then $G$ contains a $K$-almost regular subgraph $G'$ on $m\geq n^{\frac{\e-\e^2}{2+2\e}}$ vertices such that $e(G')\geq \frac{2c}{5}m^{1+\e}$ and $K=20\cdot 2^{\frac{1}{\e^2}+1}$.
\end{lemma}

We are now ready to prove Theorem \ref{thm:general 2kcycle}.

\begin{proof}[Proof of Theorem \ref{thm:general 2kcycle}]
	Let $C=C(k,s)$ be sufficiently large. By Lemma \ref{lemmaJS}, $G$ has a $K$-almost regular subgraph $G'$ on $m=\omega(1)$ vertices such that $e(G')\geq \frac{2C}{5}m^{1+1/k}$ and $K=O_k(1)$.
	$G'$ has average degree $\bar{d}\geq \frac{4C}{5}m^{1/k}$ and $\Delta(G')\leq K\bar{d}$, so by Lemma \ref{lemma:sidorenko}, we have
	$$\hom(C_{2k},G')\geq \bar{d}^{2k}\geq \frac{\bar{d}^k}{K^k}\Delta(G')^k\geq \left(\frac{4C}{5K}\right)^km\Delta(G')^k,$$
	and
	\begin{equation}
		\hom(C_{2k},G')^{\frac{1}{2k}}\geq \left(\frac{4C}{5K}\right)^{1/2}m^{\frac{1}{2k}}\Delta(G')^{1/2}. \label{eqn:bound on homc2k}
	\end{equation}
	
	\sloppy By Lemma \ref{lemma:simple with vertices} applied to the graph $G'$, the number of homomorphic $2k$-cycles $(x_1,x_2,\dots,x_{2k})$ in $G'$ such that $x_i\sim x_j$ for some $i\neq j$ is at most
	$$32k^{3/2}s^{1/2}\Delta(G')^{1/2}m^{\frac{1}{2k}}\hom(C_{2k},G')^{1-\frac{1}{2k}}.$$
	By Lemma \ref{lemma:simple with edges}, we have the same upper bound for the number of homomorphic $2k$-cycles $(x_1,x_2,\dots,x_{2k})$ in $G'$ such that $x_ix_{i+1}\approx x_jx_{j+1}$ for some $i\neq j$. Finally, Lemma \ref{lemma:simple with vertices} also implies that the number of homomorphic $2k$-cycles $(x_1,x_2,\dots,x_{2k})$ in $G'$ such that $x_i= x_j$ for some $i\neq j$ is at most
	$$32k^{3/2}\Delta(G')^{1/2}m^{\frac{1}{2k}}\hom(C_{2k},G')^{1-\frac{1}{2k}}.$$
	Indeed, we can apply that lemma with $\sim$ being the relation such that $u\sim v$ if and only if $u=v$, and take $s=1$.
	
	Combining our bounds, we see that the number of homomorphic $2k$-cycles $(x_1,x_2,\dots,x_{2k})$ in $G'$ such that $x_i=x_j$ for some $i\neq j$, $x_i\sim x_j$ for some $i\neq j$ or $x_ix_{i+1}\approx x_jx_{j+1}$ for some $i\neq j$ is at most
	\begin{equation*}
		96k^{3/2}s^{1/2}\Delta(G')^{1/2}m^{\frac{1}{2k}}\hom(C_{2k},G')^{1-\frac{1}{2k}},
	\end{equation*}
	which is at most
	\begin{equation}
		\left(\frac{5K}{4C}\right)^{1/2}96k^{3/2}s^{1/2}\hom(C_{2k},G') \label{eqn:bad cycles upper bound}
	\end{equation} by equation (\ref{eqn:bound on homc2k}). If $C$ is sufficiently large compared to $k$ and $s$, then this is less than $\hom(C_{2k},G')$ and therefore there exists a $2k$-cycle in $G'$ with the desired properties.
\end{proof}

\subsection{Rainbow cycles of length $2k$} \label{subsec:2kcycle}

We can quickly deduce Theorem \ref{thm:rainbowC2k}.

\begin{proof}[Proof of Theorem \ref{thm:rainbowC2k}]
	Let $C=C(k,1)$ be the constant provided by Theorem \ref{thm:general 2kcycle} and let $G$ be a properly edge coloured graph with $n$ vertices and at least $Cn^{1+1/k}$ edges. Define a symmetric binary relation on $E(G)$ by setting $e\approx f$ if $e$ and $f$ have the same colour. Since the colouring of $G$ is proper, for any $uv\in E(G)$ and $w\in V(G)$, there is at most one neighbour $z$ of $w$ with $uv\approx wz$. Thus, by Theorem \ref{thm:general 2kcycle} (applied with $\sim$ being the empty relation), $G$ contains a rainbow $2k$-cycle.
\end{proof}

\subsection{Colour-isomorphic cycles} \label{subsec:repeatedcycles}

In this subsection we prove Theorem \ref{thm:repeatedcycles}. Throughout the subsection, let $k$ and $r$ be fixed.

\begin{definition} \label{def:patternaux}
	Given an edge-colouring of $K_n$, define an auxiliary graph $\cG$ as follows. Let the vertex set of $\cG$ be the set of $r$-tuples of $V(K_n)$ consisting of pairwise distinct vertices. Now let $(u_1,\dots,u_r)$ and $(v_1,\dots,v_r)$ be joined by an edge if $\{u_1,\dots,u_r\}\cap \{v_1,\dots,v_r\}=\emptyset$ and the edges $u_iv_i$ have the same colour for all $1\leq i\leq r$.
\end{definition}

Call $r$-tuples $x$ and $y$ disjoint if no coordinate of $x$ is equal to any coordinate of $y$. We will prove that if $K_n$ is coloured with $o(n^{\frac{r}{r-1}\cdot \frac{k-1}{k}})$ colours, then there exists a $2k$-cycle in $\cG$ whose vertices are pairwise disjoint. Clearly, this implies that there exist $r$ colour-isomorphic, pairwise vertex-disjoint copies of $C_{2k}$.

\begin{lemma} \label{lemma:countmatchings}
	If $K_n$ is properly edge-coloured with $o(n^{\frac{r}{r-1}\cdot \frac{k-1}{k}})$ colours, then $e(\cG)=\omega(n^{r+r/k})$.
\end{lemma}

\begin{proof}
	By the convexity of the function ${x \choose r}$, the number of monochromatic $r$-matchings in $K_n$ is $\omega\left(n^{\frac{r}{r-1}\cdot \frac{k-1}{k}}\cdot (n^{2-\frac{r}{r-1}\cdot \frac{k-1}{k}})^r\right)=\omega(n^{r+r/k})$. Any monochromatic $r$-matching gives rise to at least one edge in $\cG$ and any edge in $\cG$ is counted once by this, so the statement of the lemma follows.
\end{proof}

For the rest of the proof, we fix a proper edge-colouring of $K_n$ with $o(n^{\frac{r}{r-1}\cdot \frac{k-1}{k}})$ colours and define $\cG$ as above.
Since $\cG$ has $N:=n(n-1)\dots (n-r+1)$ vertices and  $\ex(N,C_{2k})=O(N^{1+1/k})$, it is already clear by Lemma \ref{lemma:countmatchings} that $\cG$ contains a copy of $C_{2k}$. Using Theorem~\ref{thm:general 2kcycle}, we will prove that this $C_{2k}$ can be chosen in a way that the vertices are pairwise disjoint.

The following simple lemma will be needed for making sure that the conditions of Theorem~\ref{thm:general 2kcycle} are satisfied with $s=r^2$.

\begin{lemma} \label{lemma:fewintersect}
	Let $x,y\in V(\cG)$. Then the number of $z\in V(\cG)$ such that $yz\in E(\cG)$ and $x$ and $z$ are not pairwise disjoint is at most $r^2$.
\end{lemma}

\begin{proof}
	Since $x$ has $r$ coordinates, there are $r$ ways to specify which coordinate $v$ of $x$ will be a coordinate of $z$ too. Given this choice, there are $r$ ways to choose the colour of the monochromatic matching between $y$ and $z$ since it must be the colour of $uv$ for some coordinate $u$ of $y$. Given these two choices, $z$ is uniquely determined (if exists) since the colouring of $K_n$ is proper.
\end{proof}

We are now in a position to prove Theorem \ref{thm:repeatedcycles}.

\begin{proof}[Proof of Theorem \ref{thm:repeatedcycles}]
	Suppose that $K_n$ is properly edge-coloured with $o(n^{\frac{r}{r-1}\cdot \frac{k-1}{k}})$ colours.
	By Lemma~\ref{lemma:countmatchings}, we have $e(\cG)=\omega(N^{1+1/k})$, where $N=|V(\cG)|=n(n-1)\dots (n-r+1)$. Define a binary relation $\sim$ on $V(\cG)$ by setting $x\sim y$ if and only if $x$ and $y$ are not disjoint. By Lemma \ref{lemma:fewintersect}, the conditions of Theorem \ref{thm:general 2kcycle} are satisfied with $s=r^2$ ($\approx$ is chosen to be the empty relation). It follows that there is a $2k$-cycle in $\cG$ with pairwise disjoint vertices. This guarantees the existence of $r$ colour-isomorphic, pairwise vertex-disjoint copies of $C_{2k}$.
\end{proof}

\subsection{Finding good theta graphs}

Recall from the introduction that the theta graph $\theta_{k,t}$ is the union of $t$ paths of length $k$ which share the same endpoints but are pairwise internally vertex-disjoint. We can generalize Theorem \ref{thm:general 2kcycle} and find theta graphs without a conflicting pair of vertices or edges.

\begin{theorem} \label{thm:general theta}
	Let $k,t\geq 2$ and $s$ be positive integers. Then there exists a constant $C=C(k,t,s)$ with the following property. Suppose that $G=(V,E)$ is a graph with $n$ vertices and at least $Cn^{1+1/k}$ edges. Let $\sim$ be a symmetric binary relation on $V$ such that for every $u\in V$ and $v\in V$, $v$ has at most $s$ neighbours $w\in V$ which satisfy $u\sim w$. Let $\approx$ be a binary relation on $E$ such that for every $uv\in E$ and $w\in V$, $w$ has at most $s$ neighbours $z\in V$ which satisfy $uv\approx wz$. Then $G$ contains a subgraph isomorphic to $\theta_{k,t}$ which does not have two vertices related by $\sim$, nor two edges related by $\approx$.
\end{theorem}

\begin{proof}
	The beginning of the proof is almost identical to that of Theorem \ref{thm:general 2kcycle}. We take a subgraph $G'$ as in Theorem \ref{thm:general 2kcycle}. Call a homomorphic $2k$-cycle $(x_1,x_2,\dots,x_{2k})$ in $G'$ \emph{bad} if $x_i=x_j$ for some $i\neq j$, $x_i\sim x_j$ for some $i\neq j$ or $x_ix_{i+1}\approx x_jx_{j+1}$ for some $i\neq j$.
	
	By (\ref{eqn:bad cycles upper bound}),
	$$\left(\frac{5K}{4C}\right)^{1/2}96k^{3/2}s^{1/2}\hom(C_{2k},G')$$ is again an upper bound for the number of bad homomorphic $2k$-cycles $(x_1,x_2,\dots,x_{2k})$ in $G'$. Let $C$ be large enough such that this is less than $\frac{1}{t^2}\hom(C_{2k},G')$. By averaging, it follows that there exist vertices $x_1,x_{k+1}\in V$ such that among all homomorphic $2k$-cycles of the form $(x_1,x_2,\dots,x_k,x_{k+1},x_{k+2},\dots x_{2k})$ in $G'$, less than $1/t^2$ proportion are bad. Fix these vertices.
	Sample uniformly at random with replacement $t$ walks in $G'$ from $x_1$ to $x_{k+1}$: call them $x_1y_2^iy_3^i\dots y_k^ix_{k+1}$ for $1\leq i\leq t$. Note that for any $i\neq j$, $(x_1,y_2^{i},\dots ,y_k^{i},x_{k+1},y_k^{j},\dots,y_2^j)$ is a uniformly random homomorphic $2k$-cycle of the form $(x_1,\dots ,x_{k+1},\dots)$ in $G'$, so by the choice of $x_1$ and $x_{k+1}$, the probability that it is bad is less than $1/t^2$. By the union bound, with positive probability $(x_1,y_2^{i},\dots ,y_k^{i},x_{k+1},y_k^{j},\dots,y_2^j)$ is not bad for any $i\neq j$. But then the union of the walks $x_1y_2^iy_3^i\dots y_k^ix_{k+1}$ for $1\leq i\leq t$ provides a subgraph isomorphic to $\theta_{k,t}$ which has no pair of vertices related by $\sim$ and which has no pair of edges related by $\approx$.
\end{proof}

It is clear that Theorem \ref{thm:general theta} implies Theorem \ref{thm:rainbowtheta}.

\section{Rainbow cycles of arbitrary length} \label{sec:acyclic}

In this section we prove Theorem \ref{thm:acyclicrainbow}. We will use the following lemma.

\begin{lemma} \label{lemma:many2kcycles}
	Let $k\geq 2$ be an integer and let $G$ be a properly edge-coloured graph on $n$ vertices. If $\hom(C_{2k},G)> 64^{2k}k^{3k}n\Delta(G)^k$, then $G$ contains a rainbow $C_{2k}$.
\end{lemma}

\begin{proof}
	\sloppy By Lemma \ref{lemma:simple with edges}, the number of homomorphic $2k$-cycles $(x_1,\dots,x_{2k})$ in $G$ for which there exist $i\neq j$ such that $x_ix_{i+1}$ and $x_jx_{j+1}$ have the same colour is at most $32k^{3/2}\Delta(G)^{1/2}n^{\frac{1}{2k}}\hom(C_{2k},G)^{1-\frac{1}{2k}}$. By Lemma \ref{lemma:simple with vertices}, we have the same upper bound for those homomorphic $2k$-cycles $(x_1,\dots,x_{2k})$ for which there exist $i\neq j$ and $x_i=x_j$. Hence, all but at most $64k^{3/2}\Delta(G)^{1/2}n^{\frac{1}{2k}}\hom(C_{2k},G)^{1-\frac{1}{2k}}$ homomorphic $2k$-cycles provide a rainbow $2k$-cycle. But by assumption $\hom(C_{2k},G)^{\frac{1}{2k}}> 64k^{3/2}n^{\frac{1}{2k}}\Delta(G)^{1/2}$, so there exists a homomorphic $2k$-cycle which gives a rainbow subgraph isomorphic to $C_{2k}$. 
\end{proof}

Using Lemma \ref{lemma:many2kcycles}, one can show that for any $K$ there exists a constant $C$ such that any properly edge-coloured $K$-almost regular graph on $n$ vertices with at least $Cn(\log n)^3$ edges contains a rainbow cycle. Unfortunately, we think that it is not possible to find a $O(1)$-almost regular subgraph on $m=\omega(1)$ vertices with $\omega(m(\log m)^3)$ edges in an arbitrary $n$-vertex graph with $\omega(n(\log n)^3)$ edges. The next two lemmas give us a suitable subgraph for which Lemma~\ref{lemma:many2kcycles} is applied, but we lose a $\log n$ factor on the way, that is why we need  $Cn(\log n)^4$ edges in Theorem~\ref{thm:acyclicrainbow}.

\begin{lemma} \label{lemma:biregularstep1}
	Let $d$ be sufficiently large and let $G$ be a graph on $n$ vertices with average degree $d$. Then there exists a non-empty bipartite subgraph $G'$ of $G$ with parts $X$ and $Y$ such that $e(G')\geq |X|\cdot \frac{\Delta(G')}{80}$ and $e(G')\geq |Y|\cdot \frac{d}{10\log n}$.
\end{lemma}

\begin{proof}
	By passing to a suitable subgraph, we may, without loss of generality, assume that every subgraph of $G$ has average degree at most $d$. Indeed, if $\tilde{G}$ is a subgraph with average degree $\tilde{d}>d$ and $\tilde{n}$ vertices, then $\frac{\tilde{d}}{\log \tilde{n}}>\frac{d}{\log n}$.
	
	Let $A$ be the set consisting of the $\lceil n/2\rceil$ largest degree vertices in $G$ and let $B=V(G)\setminus A$.
	
	Suppose first that $e(G\lbrack B\rbrack)\geq \frac{e(G)}{10}$. Then we may partition $B$ into sets $X$ and $Y$ such that $e(G\lbrack X,Y\rbrack)\geq \frac{e(G)}{20}=\frac{nd}{40}$. Let $G'=G\lbrack X,Y\rbrack$. Any vertex in $B$ has degree at most $\frac{2e(G)}{\lceil n/2\rceil}=\frac{nd}{\lceil n/2\rceil}\leq 2d$ in $G$, so $\Delta(G')\leq 2d$. Since $|X|,|Y|\leq n/2$, $G'$ satisfies the conditions in the lemma.
	
	Hence, we may assume that $e(G\lbrack B\rbrack)< \frac{e(G)}{10}$. Suppose that $e(G\lbrack A\rbrack)\geq \frac{6e(G)}{10}$. Then $G\lbrack A\rbrack$ has larger average degree than $G$, which is a contradiction. Thus, $e(G\lbrack A\rbrack)< \frac{6e(G)}{10}$ and so $e(G\lbrack A,B\rbrack)\geq \frac{3e(G)}{10}$.
	
	Let $A_{\mathrm{low}}=\{x\in A: |N_G(x)\cap B|\leq \frac{d}{20}\}$ and let $A'=A\setminus A_{\mathrm{low}}$. Clearly, $e(G\lbrack A_{\mathrm{low}},B\rbrack)\leq n\frac{d}{20}=\frac{e(G)}{10}$, so $e(G\lbrack A',B\rbrack)\geq \frac{e(G)}{5}$. For $0\leq i\leq \lfloor \log n \rfloor$, let $A_i=\{x\in A': 2^i\leq |N_G(x)\cap B|< 2^{i+1}\}$. The sets $A_i$ partition $A'$, so there exists some $i$ such that $e(G\lbrack A_i,B\rbrack)\geq \frac{e(G\lbrack A',B\rbrack)}{\log n +1}\geq \frac{e(G)}{10\log n}=\frac{nd}{20\log n}\geq |B|\cdot \frac{d}{10\log n}$.
	
	Let $X=A_i$, $Y=B$ and $G'=G\lbrack X,Y\rbrack$. The last inequality from the previous paragraph gives that $e(G')\geq |Y|\cdot \frac{d}{10\log n}$. Since every $x\in A_i$ has $\frac{d}{20}< d_{G'}(x)<2^{i+1}$, we have $\frac{d}{20}<2^{i+1}$. But every $y\in B$ has $d_{G'}(y)\leq d_G(y)\leq 2d$, so $\Delta(G')\leq 40\cdot 2^{i+1}$. However, for every $x\in A_i$, we have $d_{G'}(x)\geq 2^{i}$, so $e(G')\geq |X|\cdot 2^{i}\geq |X|\cdot \frac{\Delta(G')}{80}$.
\end{proof}

\begin{lemma} \label{lemma:biregularstep2}
	Let $d$ be sufficiently large and let $G$ be a graph on $n$ vertices with average degree $d$. Then there exists a non-empty bipartite subgraph $G''$ of $G$ with parts $X$ and $Y$ such that for every $x\in X$, we have $d_{G''}(x)\geq \frac{\Delta(G'')}{160}$ and for every $y\in Y$, we have $d_{G''}(y)\geq \frac{d}{20\log n}$.
\end{lemma}

\begin{proof}
	By Lemma \ref{lemma:biregularstep1}, we may choose a non-empty bipartite subgraph $G'$ with parts $X'$ and $Y'$ such that $e(G')\geq |X'|\cdot \frac{\Delta(G')}{80}$ and $e(G')\geq |Y'|\cdot \frac{d}{10\log n}$. Now perform the following simple algorithm: as long as there is a vertex in $X'$ which has degree less than $\frac{\Delta(G')}{160}$ in the current graph, or there is a vertex in $Y'$ which has degree less than $\frac{d}{20\log n}$ in the current graph, then discard one such vertex. Let the final graph be $G''$ and let its parts be $X$ and $Y$. Clearly we have $d_{G''}(x)\geq \frac{\Delta(G')}{160}\geq \frac{\Delta(G'')}{160}$ for every $x\in X$ and $d_{G''}(y)\geq \frac{d}{20\log n}$ for every $y\in Y$. Finally, $G''$ is non-empty since the number of edges discarded by the algorithm is less than $|X'|\cdot \frac{\Delta(G')}{160}+|Y'|\cdot \frac{d}{20\log n}\leq e(G')$.
\end{proof}

Now we prove that the subgraph we find by Lemma \ref{lemma:biregularstep2} has many homomorphic $C_{2k}$'s.

\begin{lemma} \label{lemma:homcycles}
	Let $G$ be a bipartite graph with parts $X$ and $Y$ such that $d(x)\geq s$ for every $x\in X$ and $d(y)\geq t$ for every $y\in Y$. Then, for every positive integer $k$,
	$$\Hom(C_{2k},G)\geq s^kt^k.$$
\end{lemma}

\begin{proof}
	If $k$ is even, then $\Hom(P_k,G)\geq |X|s^{k/2}t^{k/2}$. Hence,
	\begin{align*}
		\Hom(C_{2k},G)
		&\geq \sum_{x,x'\in X} \hom_{x,x'}(P_{k},G)^2\geq \frac{1}{|X|^2}\left(\sum_{x,x'\in X} \hom_{x,x'}(P_{k},G)\right)^2\geq  \left(\frac{\Hom(P_k,G)}{|X|}\right)^2 \\
		&\geq s^kt^k.
	\end{align*}
	
	Now suppose that $k$ is odd. Without loss of generality, we may assume that $|X|s\geq |Y|t$. Note that $\Hom(P_k,G)\geq |X|s^{\frac{k+1}{2}}t^{\frac{k-1}{2}}$. Hence,
	\begin{align*}
		\Hom(C_{2k},G)
		&\geq \sum_{x\in X,y\in Y} \hom_{x,y}(P_{k},G)^2\geq \frac{1}{|X||Y|}\left(\sum_{x\in X,y\in Y} \hom_{x,y}(P_{k},G)\right)^2\geq  \frac{\Hom(P_k,G)^2}{|X||Y|} \\
		&\geq \frac{|X|}{|Y|}s^{k+1}t^{k-1}\geq s^kt^k.
	\end{align*}
	\\[-2\baselineskip]
\end{proof}

\begin{lemma} \label{lemma:cyclesinG''}
	Let $d$ be sufficiently large and let $G$ be a graph on $n$ vertices with average degree $d$. Then there exists a non-empty bipartite subgraph $G''$ of $G$ such that for every positive integer $k$,
	$$\Hom(C_{2k},G'')\geq \left(\frac{d}{20\log n}\right)^k\left(\frac{\Delta(G'')}{160}\right)^k.$$
\end{lemma}

\begin{proof}
	This follows immediately from Lemma \ref{lemma:biregularstep2} and Lemma \ref{lemma:homcycles}.
\end{proof}

\begin{proof}[Proof of Theorem \ref{thm:acyclicrainbow}]
	Let $n$ be sufficiently large and let $G$ be a properly edge-coloured graph on $n$ vertices with at least $Cn(\log n)^4$ edges, where $C=2^{100}$. Let $k=\lfloor \log n\rfloor$.
	
	By Lemma \ref{lemma:cyclesinG''}, $G$ has a non-empty bipartite subgraph $G''$ such that
	$$\Hom(C_{2k},G'')\geq \left(\frac{C}{10}(\log n)^3\right)^k\left(\frac{\Delta(G'')}{160}\right)^k\geq 2^{50k} k^{3k} \Delta(G'')^k> 64^{2k} k^{3k} n\Delta(G'')^k.$$
	Then, by Lemma \ref{lemma:many2kcycles}, $G''$ contains a rainbow cycle. It has even length because $G''$ is bipartite.
\end{proof}

\section{Blow-ups of cycles} \label{sec:blownupcycles}

In this section we prove Theorem \ref{thm:turanblowup} and Theorem \ref{thm:turanblowup2k}.

\begin{definition} \label{def:blowupaux}
	Given a graph $G$, define an auxiliary graph $\cG_0$ as follows. Let the vertex set of $\cG_0$ be the set of $r$-vertex subsets of $V(G)$, i.e. let $V(\cG_0)=V(G)^{(r)}$. Now let $U$ and $V$ be joined by an edge in $\cG_0$ if $U\cap V=\emptyset$ and $uv\in E(G)$ for every $u\in U$ and $v\in V$.
\end{definition}

For the rest of the proof, we fix a positive integer $r$ and a graph $G$, and define $\cG_0$ as above.
In order to find a copy of $C_{2k}[r]$ in $G$, we need to find a copy of $C_{2k}$ in $\cG_0$ in which the vertices are disjoint as subsets of $V(G)$.
The next lemma will be used for making sure that we can apply Lemma \ref{lemma:bipartite with vertices} with $s_1$ and $s_2$ not too large.

\begin{lemma} \label{lemma:blowupfewintersect}
	Let $x,y\in V(\cG_0)$. Then the number of $z\in V(\cG_0)$ such that $yz\in E(\cG_0)$ and $z\cap x\neq \emptyset$ is at most $r^{r+1}d_{\cG_0}(y)^{1-1/r}$.
\end{lemma}

\begin{proof}
	There are $r$ ways to choose the element of $x$ that should belong to $z$, so it suffices to prove that for any $v\in V(G)$, the number of neighbours of $y$ in $\cG_0$ that contain $v$ is at most $r^rd_{\cG_0}(y)^{1-1/r}$. Let $d$ be the size of the common neighbourhood (in $G$) of the vertices in $y$. We may assume that $v$ belongs to this common neighbourhood. Now there are ${d-1 \choose r-1}$ ways to choose the $r-1$ vertices in $z$ that are different from $v$. Since ${d-1 \choose r-1}\leq r^r{d \choose r}^{1-1/r}= r^rd_{\cG_0}(y)^{1-1/r}$, the proof is complete.
\end{proof}

The next lemma is the only place in the paper where we apply Lemma \ref{lemma:bipartite with vertices} with $\ell\neq 0$. It turns out that taking $\ell=1$ here leads to a slightly better bound for $\ex(n,C_{2k}[r])$.

\begin{lemma} \label{lemma:blowupcountbad}
	Let $k\geq 2$ be an integer and let $\cG$ be a bipartite subgraph of $\cG_0$ on $m$ vertices with parts $X_1$ and $X_2$ such that every $x\in X_1$ has $d_{\cG_0}(x)\leq D_1$ and every $x\in X_2$ has $d_{\cG_0}(x)\leq D_2$, where $D_1\leq D_2$. Then the number of homomorphic copies of $C_{2k}$ in $\cG$ in which the vertices are not pairwise disjoint (as subsets of $V(G)$) is at most
	\begin{enumerate}[label=(\roman*)]
		\item $32k^{3/2}r^{\frac{r+1}{2}} (D_1^{1-1/r}D_2)^{1/2}m^{\frac{1}{2k}}\hom(C_{2k},\cG)^{1-\frac{1}{2k}}$, and also at most
		\item $32k^{3/2}r^{\frac{r+1}{2}} (D_1^{1-1/r}D_2)^{1/2}(mD_2)^{\frac{1}{2k-2}}\hom(C_{2k},\cG)^{1-\frac{1}{2k-2}}$.
	\end{enumerate}
\end{lemma}

\begin{proof}
	Define a symmetric binary relation $\sim$ over $V(\cG)$ by taking $x\sim y$ if and only if $x\cap y\neq \emptyset$. By Lemma \ref{lemma:blowupfewintersect}, the conditions of Lemma \ref{lemma:bipartite with vertices} are satisfied with $\Delta_1=D_1$, $\Delta_2=D_2$, $s_1=r^{r+1}D_1^{1-1/r}$ and $s_2=r^{r+1}D_2^{1-1/r}$. Note that $\max(\Delta_1 s_2,\Delta_2 s_1)=r^{r+1}D_1^{1-1/r}D_2$. Hence, (i) follows from applying Lemma \ref{lemma:bipartite with vertices} with $\ell=0$ and noting that $\hom(C_0,\cG)=|V(\cG)|=m$. Moreover, (ii) follows from applying Lemma \ref{lemma:bipartite with vertices} with $\ell=1$ and noting that $\hom(C_2,\cG)\leq |V(\cG)|\Delta(\cG)\leq mD_2$.
\end{proof}

Now we want to find a bipartite subgraph $\cG$ in $\cG_0$ which has many homomorphic cycles but whose vertices have not too large degree in $\cG_0$.

\begin{lemma} \label{lemma:blowupbiregular}
	Let $\cG_0$ have average degree $d>0$. Then there exist $D_1,D_2\geq \frac{d}{4}$ and a non-empty bipartite subgraph $\cG$ in $\cG_0$ with parts $X_1$ and $X_2$ such that for every $x\in X_1$, we have $d_{\cG}(x)\geq \frac{D_1}{256r^2(\log n)^2}$ and $d_{\cG_0}(x)\leq D_1$, and for every $x\in X_2$, we have $d_{\cG}(x)\geq \frac{D_2}{256r^2(\log n)^2}$ and $d_{\cG_0}(x)\leq D_2$.
\end{lemma}

\begin{proof}
	Let $N$ and $e$ denote the number of vertices and edges in $\cG_0$, respectively. Observe that the number of edges in $\cG_0$ incident to vertices of degree at most $d/4$ is at most $Nd/4=e/2$. Hence, a random partitioning of all vertices with degree at least $d/4$ shows that there exist disjoint sets $A$ and $B$ in $V(\cG_0)$ such that for every $v\in A\cup B$ we have $d_{\cG_0}(v)\geq d/4$ and the number of edges in $\cG_0[A,B]$ is at least $e/4$. For each $1\leq i\leq \lceil r\log n\rceil$, let $A_i=\{v\in A: 2^{i-1}\leq d_{\cG_0}(v)<2^{i}\}$ and let $B_i=\{v\in B: 2^{i-1}\leq d_{\cG_0}(v)<2^{i}\}$. Note that the $A_i$'s partition $A$. Indeed, $\Delta(\cG_0)\leq {n \choose r}\leq n^r$. Similarly, the $B_i$'s partition $B$. Hence, there exist $i,j$ such that $e(\cG_0[A_i,B_j])\geq \frac{e}{4\lceil r\log n\rceil^2}\geq \frac{e}{16r^2(\log n)^2}$.
	
	Note that $|A_i|2^{i-1}\leq 2e(\cG_0)=2e$, so $|A_i|\leq \frac{2e}{2^{i-1}}$. Thus, the average degree of the vertices in $A_i$ in the graph $\cG_0[A_i,B_j]$ is at least $\frac{2^{i-1}}{32r^2(\log n)^2}$. Similarly, the average degree of the vertices in $B_j$ in the same graph is at least $\frac{2^{j-1}}{32r^2(\log n)^2}$. Thus, by a standard vertex removal argument, there exist non-empty $X_1\subset A_i$ and $X_2\subset B_j$ such that for $\cG=\cG_0[X_1,X_2]$, we have $d_{\cG}(x)\geq \frac{2^{i-1}}{128r^2(\log n)^2}$ for every $x\in X_1$ and $d_{\cG}(x)\geq \frac{2^{j-1}}{128r^2(\log n)^2}$ for every $x\in X_2$. Take $D_1=2^{i}$ and $D_2=2^j$. Since $d/4\leq d_{\cG_0}(v)<2^{i}$ holds for every $v\in X_1\subset A$, we have $D_1>d/4$. Similarly, $D_2>d/4$.
\end{proof}

The following supersaturation result guarantees that $\cG_0$ has enough edges, and is our final ingredient to the proof of Theorem \ref{thm:turanblowup}.

\begin{lemma}[Erd\H os--Simonovits \cite{ESi83}] \label{lemma:supersaturation}
	There exist positive constants $c=c(r),\gamma=\gamma(r)$ such that any graph on $n$ vertices with $e>c\cdot n^{2-\frac{1}{r}}$ edges contains at least $\gamma \frac{e^{r^2}}{n^{2r^2-2r}}$ copies of $K_{r,r}$.
\end{lemma}

\begin{proof}[Proof of Theorem \ref{thm:turanblowup}]
	Let $G$ be an $n$-vertex graph with $\omega(n^{2-1/r}(\log n)^{7/r})$ edges. We will prove that if $n$ is sufficiently large, then $G$ contains an $r$-blownup cycle. By Lemma~\ref{lemma:supersaturation}, $\cG_0$ has $\omega(n^{r}(\log n)^{7r})$ edges, so it has average degree $\omega((\log n)^{7r})$. By Lemma \ref{lemma:blowupbiregular}, there exist $D_1,D_2=\omega((\log n)^{7r})$ and a non-empty bipartite subgraph $\cG$ in $\cG_0$ with parts $X_1$ and $X_2$ such that for every $x\in X_1$, we have $d_{\cG}(x)\geq \frac{D_1}{256r^2(\log n)^2}$ and $d_{\cG_0}(x)\leq D_1$, and for every $x\in X_2$, we have $d_{\cG}(x)\geq \frac{D_2}{256r^2(\log n)^2}$ and $d_{\cG_0}(x)\leq D_2$. Without loss of generality, we may assume that $D_1\leq D_2$.
	
	By Lemma \ref{lemma:homcycles}, for every positive integer $k$ we have
	$$\hom(C_{2k},\cG)\geq \left(\frac{D_1}{256r^2(\log n)^2}\right)^k\left(\frac{D_2}{256r^2(\log n)^2}\right)^k=\left(\frac{D_1^{1/r}}{2^{16}r^4(\log n)^4}\right)^k(D_1^{1-1/r}D_2)^k.$$
	Let $k=\lfloor \log n\rfloor$ and let $m$ be the number of vertices in $\cG$. Then $m\leq \binom{n}{r}$, so since $D_1=\omega((\log n)^{7r})$, we have
	$$\left(\frac{D_1^{1/r}}{2^{16}r^4(\log n)^4}\right)^k\geq L^kmk^{3k}$$
	for some $L=\omega(1)$, which gives
	$$\hom(C_{2k},\cG)\geq L^kmk^{3k}(D_1^{1-1/r} D_2)^k.$$
	Therefore
	$$\hom(C_{2k},\cG)^{\frac{1}{2k}}\geq L^{1/2}m^{\frac{1}{2k}}k^{3/2}(D_1^{1-1/r} D_2)^{1/2},$$
	so by Lemma \ref{lemma:blowupcountbad} (i), at most $o(\hom(C_{2k},\cG))$ homomorphic $2k$-cycles in $\cG$ have a pair of vertices which are not disjoint.
	Thus, for sufficiently large $n$, there exists a homomorphic copy of $C_{2k}$ in $\cG$ whose vertices are pairwise disjoint subsets of $V(G)$. This gives a $C_{2k}[r]$ in $G$.
\end{proof}

We will now prove Theorem \ref{thm:turanblowup2k}.

\begin{proof}[Proof of Theorem \ref{thm:turanblowup2k}]
	\sloppy Let $G$ be a graph with $\omega\left(n^{2-\frac{1}{r}+\frac{1}{k+r-1}}(\log n)^{\frac{4k}{r(k+r-1)}}\right)$ edges. By Lemma~\ref{lemma:supersaturation}, $\cG_0$ has $\omega\left(n^{r+\frac{r^2}{k+r-1}}(\log n)^{\frac{4kr}{k+r-1}}\right)$ edges, so it has average degree $\omega\left(n^{\frac{r^2}{k+r-1}}(\log n)^{\frac{4kr}{k+r-1}}\right)$. By Lemma \ref{lemma:blowupbiregular}, $\cG_0$ has a bipartite subgraph $\cG$ with parts $X_1$ and $X_2$ such that for every $x\in X_i$ we have $d_{\cG}(x)\geq \frac{D_i}{256r^2(\log n)^2}$ and $d_{\cG_0}(x)\leq D_i$, where $D_i=\omega\left(n^{\frac{r^2}{k+r-1}}(\log n)^{\frac{4kr}{k+r-1}}\right)$. Without loss of generality, $D_2\geq D_1$. Using Lemma \ref{lemma:homcycles}, we have $$\hom(C_{2k},\cG)\geq \Omega\left(\frac{D_1^kD_2^k}{(\log n)^{4k}}\right)\geq \omega\left((D_1^{1-1/r} D_2)^{k-1} {n \choose r}D_2\right).$$
	Since $\cG$ has $m\leq \binom{n}{r}$ vertices,
	$$\hom(C_{2k},\cG)^{\frac{1}{2k-2}}\geq \omega\left((D_1^{1-1/r} D_2)^{1/2} (mD_2)^{\frac{1}{2k-2}}\right).$$
	Therefore Lemma \ref{lemma:blowupcountbad} (ii) shows that the number of homomorphic $2k$-cycles in $\cG$ whose vertices are not pairwise disjoint subsets of $V(G)$ is $o(\hom(C_{2k},\cG))$. Thus, for sufficiently large $n$, there exists a $2k$-cycle in $\cG$ with pairwise disjoint vertices, yielding a copy of $C_{2k}[r]$ in $G$.
\end{proof}

\section{Concluding remarks} \label{sec:rainbowremarks}

{\bf Rainbow cycles.} We have shown that for a sufficiently large constant $C$, any properly edge-coloured $n$-vertex graph with at least $Cn(\log n)^4$ edges contains a rainbow cycle. However, the best known construction of a graph without a rainbow cycle has only $\Theta(n\log n)$ edges. One such example, found by Keevash, Mubayi, Sudakov and Verstra\"ete \cite{KMSV03}, is the $m$-dimensional cube whose vertices are the subsets of $\{1,2,\dots,m\}$ where $A$ is joined to $A\setminus \{i\}$ for every $i\in A$. The colour of the edge between $A$ and $A\setminus \{i\}$ is $i$. This graph has $2^m$ vertices and $\frac{1}{2}m2^m$ edges and it has no rainbow cycle. Examples with more than $0.58n\log n$ edges were also found by Keevash, Mubayi, Sudakov and Verstra\"ete.

\vspace{3mm}
\noindent
{\bf Colour-isomorphic cycles.} Recall that $f_r(n,H)$ is the smallest number $C$ so that there is a proper edge-colouring of $K_n$ with $C$ colours containing no $r$ vertex-disjoint colour-isomorphic copies of $H$. We have shown that $f_r(n,C_{2k})=\Omega\left(n^{\frac{r}{r-1}\cdot \frac{k-1}{k}}\right)$. Note that our result becomes trivial when $r\geq k$ since $f_{r}(n,H)\geq n-1$ holds for any $r$ and $H$ (as any proper colouring of $K_n$ must use at least $n-1$ colours).

The best general upper bound comes from the probabilistic construction that is used in Theorem \ref{thm:ConlonTyomkyn} and says that $f_r(n,C_{2k})=O\left(n^{\frac{r}{r-1}-\frac{1}{(r-1)k}}\right)$. Another result of Conlon and Tyomkyn {\cite[Theorem 1.4]{CT20}}, proved by a variant of Bukh's random algebraic method \cite{Bukh15}, states that if $H$ contains a cycle, then there exists $r$ such that $f_r(n,H)=O(n)$. It would be interesting to decide what the smallest such $r$ is when $H=C_{2k}$. Our result shows that we must have $r\geq k$. This question was studied in the case $H=C_4$ by Xu, Zhang, Jing and Ge \cite{XZJG20}, who showed that $f_r(n,C_4)=\Theta(n)$ for any $r\geq 3$.

\vspace{3mm}
\noindent
{\bf The Erd\H os--Gy\'arf\'as function.} For positive integers $n$, $p$ and $2\leq q\leq {p \choose 2}$, the Erd\H os-Gy\'arf\'as function $g(n,p,q)$ is defined to be the smallest $C$ such that there exists a (not necessarily proper) colouring of the edges of $K_n$ with $C$ colours such that every induced subgraph on $p$ vertices receives at least $q$ colours. A variant of our Theorem \ref{thm:repeatedcycles} can be used to give a good lower bound for this function when $q$ is close to ${p \choose 2}$. Indeed, assume that $p=2kr$ and $q={p \choose 2}-(r-1)2k+1={p \choose 2}-p+2k+1$ for some $r,k\geq 2$. If we can find $r$ vertex-disjoint colour-isomorphic cycles of length $2k$, then the $p$ vertices of these cycles induce at most ${p \choose 2}-(r-1)2k<q$ colours. Note that the proof of Theorem \ref{thm:repeatedcycles} can be adapted to the case where the edge-colouring is not necessarily proper, but every vertex is incident to at most $O(1)$ edges of any given colour. Now if we have an arbitrary edge-colouring of $K_n$, then either every vertex is incident to at most $2kr-2$ edges of any given colour, or we can choose vertices $u_0,u_1,\dots,u_{2kr-1}$ such that the edges $u_0u_i$ are of the same colour for every $1\leq i\leq 2kr-1$. In this latter case, we have $p$ vertices which induce at most ${p \choose 2}-p+2<q$ colours. In the former case, we can use the strengthened version of Theorem \ref{thm:repeatedcycles}. We obtain the following result.

\begin{theorem} \label{thm:erdosgyarfas}
	For any integers $r,k\geq 2$,
	$$g\left(n,2kr,{2kr \choose 2}-(r-1)2k+1\right)=\Omega(n^{\frac{r}{r-1}\cdot \frac{k-1}{k}}).$$
\end{theorem}

This generalises a recent result of Fish, Pohoata and Sheffer {\cite[Theorem 1.1]{FPS20}}, which is Theorem~\ref{thm:erdosgyarfas} in the special case $r=2$.

\vspace{3mm}
\noindent
{\bf Blow-ups of cycles.} We have shown that $\ex(n,\mathcal{C}[r])=O(n^{2-1/r}(\log n)^{7/r})$. On the other hand, a random graph with edge probabilities $p=\frac{n^{-1/r}}{2}$ contains no $r$-blownup cycles with probability at least $1/2$, so $\ex(n,\mathcal{C}[r])=\Omega(n^{2-1/r})$. We pose the following question.

\begin{question}
	Let $r\geq 2$. Is it true that
	$\ex(n,\mathcal{C}[r])=\Theta(n^{2-1/r})$?
\end{question}

Finally, regarding a single forbidden blownup cycle, we reiterate the conjecture of Grzesik, Janzer and Nagy \cite{GJN19} stating that $\ex(n,C_{2k}[r])=O(n^{2-\frac{1}{r}+\frac{1}{kr}})$.

\vspace{3mm}
\noindent
{\bf Tur\'an number of $s$-regular graphs.} For any even $s$ and any $\d>0$, we have found an $s$-regular graph $H$ (namely $H=C_{2k}\lbrack s/2\rbrack$ for sufficiently large $k$) such that $\ex(n,H)=O(n^{2-\frac{2}{s}+\d})$. We believe that such graphs exist for odd $s$ too.

\begin{conjecture} \label{con:regturan}
	Let $s\geq 3$ be odd and let $\d>0$. Then there exists an $s$-regular graph $H$ such that $\ex(n,H)=O(n^{2-\frac{2}{s}+\d})$.
\end{conjecture}

We think that our methods, combined with ideas from \cite{JN17}, may be useful for proving this conjecture. As we have remarked in the introduction, the probabilistic deletion method shows that if $H$ is a bipartite graph with minimum degree $s\geq 2$, then there exists $\e=\e(H)>0$ such that $\ex(n,H)=\Omega(n^{2-\frac{2}{s}+\e})$. This shows that Conjecture \ref{con:regturan} is tight (if true).

\section*{Acknowledgement}

I am grateful to Istv\'an Tomon for pointing out (a variant of) Lemma \ref{lemma:hom inequality} which simplified the proofs compared to an earlier version of the paper. I thank Barnab\'as Janzer for helpful remarks on a draft. I am also grateful to David Conlon for drawing my attention to Conjecture~\ref{con:nondegenerate} and to Cosmin Pohoata for pointing out the connection with the Erd\H os--Gy\'arf\'as function. Finally, I thank the anonymous referee for many helpful comments which substantially improved the presentation of the paper.

\bibliographystyle{abbrv}
\bibliography{bibliography}

\end{document}